\documentclass{amsart}

\usepackage[all]{xy}

\usepackage{amsmath,amssymb,amsfonts,amsthm,graphicx,xspace}

\newcommand*{\C}{\ensuremath{\mathbb C}\xspace}
\newcommand*{\Do}{\ensuremath{\boldsymbol D}\xspace}
\newcommand*{\Dou}{\ensuremath{\boldsymbol {Dou}}\xspace}
\newcommand*{\Ho}{\ensuremath{\mathfrak H}\xspace}
\newcommand*{\E}{\ensuremath{\mathcal E}\xspace}
\newcommand*{\F}{\ensuremath{\mathcal F}\xspace}
\newcommand*{\Gf}{\ensuremath{\mathfrak G}\xspace}
\newcommand*{\Ii}{\ensuremath{\mathcal I}\xspace}
\newcommand*{\M}{\ensuremath{\mathcal M}\xspace}
\newcommand*{\m}{\ensuremath{\mathfrak m}\xspace}
\newcommand*{\Nc}{\ensuremath{\mathcal N}\xspace}
\newcommand*{\Oo}{\ensuremath{\mathcal O}\xspace}
\newcommand*{\PP}{\ensuremath{\mathbb P}\xspace}

\newcommand*{\Sc}{\ensuremath{\mathcal S}\xspace}
\newcommand*{\Vv}{\ensuremath{\mathfrak V}\xspace}
\newcommand*{\Z}{\ensuremath{\mathbb Z}\xspace}

\DeclareMathOperator{\Ann}{Ann}
\DeclareMathOperator{\Hom}{Hom}
\DeclareMathOperator{\Ker}{Ker}
\DeclareMathOperator{\pr}{pr}
\DeclareMathOperator{\rank}{rank}
\DeclareMathOperator{\Spec}{Spec}
\DeclareMathOperator{\Sym}{Sym}
\DeclareMathOperator{\supp}{supp}
\DeclareMathOperator{\trdeg}{tr.deg}

\let\eps\varepsilon
\let\ph\varphi

\newcommand*{\lst}[3][1]{\ensuremath{#2_{#1}, \ldots, #2_{#3}}\xspace}

\newtheorem{proposition}{Proposition}[section]
\newtheorem{theorem}[proposition]{Theorem}

\newtheorem{lemma}[proposition]{Lemma}

\theoremstyle{remark}

\theoremstyle{definition}
\newtheorem{definition}[proposition]{Definition}

\newtheorem{Not}[proposition]{Notation}

\title[Meromorphic functions on neighborhoods of rational
curves]{On fields of meromorphic functions on neighborhoods of rational
curves}

\author{Serge Lvovski}

\address{National Research University Higher School of Economics,
  Russian Federation 
  \hfil\break
Scientific Research Institute for System Analysis of the National
Research Centre ``Kurchatov Institute'' 
}
\email{lvovski@gmail.com}



\thanks{This study was partially supported by the HSE
  University Basic Research Program and by the SRISA project
  FNEF-2024-0001 (Reg. No 1023032100070-3-1.2.1).}

\begin{document}

\begin{abstract}
Suppose that $F$ is a smooth and connected complex surface (not
necessarily compact) containing a smooth rational curve with positive
self-intersection. We prove that if there exists a non-constant
meromorphic function on~$F$, then the field of meromorphic functions
on~$F$ is isomorphic to the field of rational functions in one or two
variables over~\C.
\end{abstract}


\maketitle

\section{Introduction}

It is well known (see
Proposition~\ref{fin.gen} below for references) that the field of
meromorphic functions on a $2$-dimensional neighborhood of the Riemann
sphere with positive self-intersection is a finitely generated
extension of~\C, of transcendence degree at most~$2$. In recent
papers~\cite{FLL2022,FL-L-S,Lvovski} examples of such neighborhoods
were constructed for which this transcendence degree assumes all
values from~$0$ through~$2$ (in particular, examples of
non-algebraizable neighborhoods with transcendence degree~$2$ were
found).

Now it seems natural to ask what fields may occur as such fields of
meromorphic functions (in the case of transcendence degree $1$ or~$2$,
of course). It turns out that the answer to this question is
simple
and somehow disappointing.
To wit, the main results of the paper are as
follows.

\begin{proposition}\label{main}
Suppose that $F$ is a non-singular connected complex surface and that
there exists a curve $C\subset F$, $C\cong\PP^1$, such that $(C\cdot
C)>0$. Let \M be the field of meromorphic functions on~$F$. 

If the transcendence degree of \M over \C is at least~$2$, then
$\M\cong\C(T_1,T_2)$ \textup(the field of rational functions\textup).
\end{proposition}

\begin{proposition}\label{main1}
Suppose that $F$ is a non-singular connected complex surface and that
there exists a curve $C\subset F$, $C\cong\PP^1$, such that $(C\cdot
C)>0$. Let \M be the field of meromorphic functions on~$F$. 

If the transcendence degree of \M over \C is $1$, then
$\M\cong\C(T)$ \textup(the field of rational functions\textup).
\end{proposition}




Summing up, if $F$ is a smooth and connected complex surface contaning
a copy of the Riemann sphere with positive self-intersetion, then the
field of meromorphic functions on~$F$ is isomorphic to either~\C or
$\C(T)$ or $\C(T_1,T_2)$.

Thus, the field of meromorphic functions without any additional
structure cannot serve as an invariant that would help to classify
neighborhoods of rational curves with positive self-intersection.

Proposition~\ref{main1} agrees with the example from~\cite[Section 3.2]{FL-L-S}.

The proof of 
Propositions~\ref{main} and~\ref{main1} are 
based on the study of
(embedded) deformations of the curve~$C\subset F$. Properties of such
deformations are well known in the algebraic context; the classical paper~\cite{Kodaira}
implies a complete description of deformations of rational curves on
arbitrary smooth complex surfaces, but this paper does not contain a description of deformations of rational curves passing through given points;
I prove the necessary facts
(Propositions~\ref{def1} and~\ref{def2}) in the ad hoc manner,
using a result of Savelyev~\cite{Savelyev}.

The paper is organized as follows. In Section~\ref{sec:def:gen} we
recall, following Douady~\cite{Douady}, general facts on deformations
of compact analytic subspaces in a given analytic space. In
Section~\ref{sec:deform:rational} we prove some pretty natural results
on deformations of smooth rational curves in smooth (and not
necessarily compact) complex surfaces; the results of this section do
not claim much novelty. In Section~\ref{sec:good} we establish some
more specific properties of deformations of rational
curves on surfaces. Finally, in Section~\ref{sec:Castelnuovo}
(resp.~\ref{sec:trdeg=1}) we prove
Proposition~\ref{main} (resp.~\ref{main1}).

\subsection*{Acknowledgements}
I am grateful to Ekaterina Amerik 
for discussions and to the anonymous referee for numerous useful
suggestions. 

\section*{Notation and conventions}

All topological terms refer to the classical topology unless specified
otherwise. By coherent sheaves we mean analytic coherent sheaves.

If $X$ is a connected complex manifold, then $\M(X)$ is the field of
meromorphic functions on~$X$.

If $Y$ is a complex submanifold of a complex manifold~$X$, then the
normal bundle to $Y$ in~$X$ is denoted by~$\Nc_{Y|X}$.

Our notation for the $n$-dimensional complex projective space
is~$\PP^n$.

The projectivization $\PP(E)$ of a linear space $E$ is the set of
lines in~$E$, not of hyperplanes.

If $C_1$ and $C_2$ are compact Riemann surfaces embedded in a smooth
complex surface~$F$, then their intersection index is denoted by
$(C_1\cdot C_2)$.

If $C$ is a Riemann surface isomorphic to $\PP^1$ and $n\in\Z$, then
$\Oo_C(n)$ stands for the line bundle aka invertible sheaf of
degree~$n$ on~$C$.

By \emph{Veronese curve} $C_d\subset\PP^d$ we mean the image of the
mapping $\PP^1\to\PP^d$ defined by the formula
$(z_0:z_1)\mapsto(z_0^d: z_0^{d-1}z_1:\dots:z_1^d)$. 

Analytic spaces are allowed to have nilpotents in their structure sheaves
(however, analytic spaces with nilpotents will be acting mostly behind the
scenes). If $X$ is an analytic space, then
the analytic space obtained from $X$ by quotienting out the nilpotents
is denoted by~$X_{\mathrm{red}}$.

If $X$ is an analytic space and $x\in X$, then $T_xX$ is the Zariski
tangent space to $X$ at~$x$ (i.e., $T_xX=(\m_x/\m_x^2)^*$, where
$\m_x$ is the maximal ideal of the local ring~$\Oo_{X,x}$). 

In the last two sections we use meromorphic mappings (which will be denoted
by dashed arrows). For the general definition we refer the reader
to~\cite[page~75]{AndreottiStoll} (one caveat: a meromorphic function
on a smooth complex manifold~$X$ \emph{is not}, in general, a
meromorphic mapping from~$X$ to~\C); for
our purposes it suffices to keep in mind two facts concerning
them. First, if $F\colon X\dasharrow Y$ is a meromorphic mapping,
where $X$ is a complex manifold, then the indeterminacy locus of~$F$
is an analytic subset in~$X$ of codimension at least~$2$. Second, if
$X$ is a connected complex manifold and \lst[0]fn are meromorphic
functions on~$X$ of which not all are identically zero, then the
formula $x\mapsto (f_0(x):\dots f_n(x))$ defines a meromorphic mapping
from~$X$ to~$\PP^n$.

\section{Deformations: generalities}\label{sec:def:gen}

In this section we recall (briefly and without proofs) the general
theory (see~\cite{Douady} for details). Suppose that $F$ is an analytic
space (in the applications we have in mind $F$ will be a smooth
complex surface). Then there exists the \emph{Douady space}~$\Do(F)$,
which parametrizes all the compact analytic subspaces of~$F$. This
means the following.

For any analytic space~$B$, a \emph{family of compact analytic subspaces
  of~$F$ with the base~$B$} is a closed analytic subspace~$\Ho\subset
B\times F$ that is proper and flat over~$B$. Now the Douady
space~$\Do(F)$ comes equipped with the \emph{universal family}
$\Ho(F)\subset \Do(F)\times F$ of
subspaces of~$F$ over $\Do(F)$, which satisfies the following
property: for any family over an analytic space~$B$ there
exists a unique morphism $B\to\Do(F)$ such that the family over~$B$ is
induced, via this morphism, from the universal family
over~$\Do(F)$. Applying this definition to the case in which $B$ is a
point (hence, a family over~$B$ is just an individual compact analytic
subspace of~$F$), one sees that there is a 1--1 correspondence
between compact analytic subspaces of~$F$ and fibers of the projection
$\Ho(F)\to\Do(F)$.

At this point one has to say that the Douady space is not an analytic
space: it is a more general object, which Douady calls a Banach
analytic space. However, every point $a\in\Do(F)$ has a neighborhood
$\Delta\ni a$ that is isomorphic to an analytic space in the usual sense.

This construction can be generalized as follows. If \E is a coherent
analytic sheaf on $F$, then there exists a Banach analytic space
$\Dou(\E)$ parametrizing coherent subsheaves $\Sc\subset\E$ such that
the quotient $\E/\Sc$ has compact support. To be more precise, a
family of subsheaves of \E with base~$B$ is a coherent subsheaf
$\varSigma\subset \pr_2^*\E$ on $B\times F$ such that
$\pr_2^*\E/\varSigma$ is flat over~$B$ and
$\supp(\pr_2^*\E/\varSigma)$ is proper over~$B$, and there is a
universal family of subsheaves of~\E over $\Dou(\E)$.

The space
$\Dou(\E)$ is also locally isomorphic to an analytic space. If one puts
$\E=\Oo_F$ in this construction, one obtains a canonical isomorphism
$\Do(F)\cong\Dou(\Oo_F)$.

If $a\in\Dou(\E)$ is a point corresponding to the subsheaf
$\Sc\subset\E$, then one can define the Zariski tangent
space~$T_a\Dou(\E)$ to $\Dou(\E)$ at~$a$ as $T_a\Delta$, where
$\Delta\subset\Dou(\E)$ is any neighborhood of~$a$ that is isomorphic
to an analytic space.  This Zariski tangent space is canonically
isomorphic to $\Hom(\Sc,\E/\Sc)$ (see \cite[Section 9.1, Remarque
  3]{Douady}).

If a coherent sheaf \E is a subsheaf of a coherent sheaf~\F and if
$\F/\E$ has compact support, then $\Dou(\E)$ is naturally embedded in
$\Dou(\F)$ (a subsheaf of \E can be regarded as a subsheaf
of~\F). This embedding induces injective homomorphisms of Zariski
tangent spaces. Indeed, let $\Spec\C[\eps]/(\eps^2)$ be the analytic
space consisting of one point such that the ring of functions
is~$\C[\eps]/(\eps^2)$.  Then $T_a\Dou(\E)$, as a set, is canonically
bijective to the set of families of subsheaves of \E over the base
$\Spec\C[\eps]/(\eps^2)$ (ibid.).  If $\supp(\F/\E)$ is compact, any
family of subsheaves of \E (over an arbitrary base) is automatically a
family of subsheaves of~\F, and different families of subsheaves of
\E, being different subsheaves of $\pr_2^*\E$, are ipso facto
different subsheaves of $\pr_2^*\F$.

In the sequel we will be using the following notation.

\begin{Not}\label{not:C_a}
If $F$ is a complex manifold and $a\in\Do(F)$, then $C_a$ stands for
the analytic subspace of~$F$ corresponding to~$a$.

Similarly, if \E is a coherent sheaf on a complex manifold and
$a\in\Dou(\E)$, then $\Sc_a$ is the subsheaf in \E corresponding to~$a$.
\end{Not}

\section{Deformations of rational curves}\label{sec:deform:rational}

In this section we state and prove two auxiliary results concerning
deformations of smooth rational curves on complex surfaces. These
results are well known for deformations of curves on which no restrictions are imposed. For example, Proposition~\ref{def1}
below follows immediately from the main result of~\cite{Kodaira}, and its
 algebraic-geometric counterpart (for smooth
algebraic surfaces over a field of characteristic zero) follows
immediately from the theorem in Lecture~23 of~\cite{Mumford}. However,
I did not manage to find a suitable reference for deformations of curves passing through given points.

We will be using the general theory from Section~\ref{sec:def:gen} in
the following setting. $F$ will always be a smooth and connected
complex surface, $C\subset F$ will be a complex submanifold isomorphic
to $\PP^1$ (the Riemann sphere), and we will always assume that the
self-intersection index $d=(C\cdot C)$ is non-negative. By $\Do(F,C)$
we will mean an unspecified open subset of $\Do(F)$ that contains the
point corresponding to $C\subset F$ and is isomorphic to an analytic
space. The reader will check that this indeterminacy of definition
does not affect the arguments that follow.

Moreover, suppose that $S=\{\lst pm\}\subset C$ is a subset of
cardinality $m\le d=(C\cdot C)$. Let $\Ii_S\subset\Oo_F$ be the ideal
sheaf of the analytic subset $S\subset F$. If $\Ii\subset \Ii_S$ is a
coherent subsheaf, then it follows from the exact sequence
\begin{equation*}
  0\to \Ii_S/I\to\Oo_F/\Ii\to \Oo_F/\Ii_S\to 0
\end{equation*}
that $\supp(\Oo_F/\Ii)= \supp(\Ii_S/I)\cup S$, so
$\supp(\Ii_S/I)$ is compact if and only if $\supp(\Oo_F/\Ii)$ is
compact. Hence, the Douady space $\Dou(\Ii_S)$ parametrizes the (ideal
sheaves of) compact analytic subspaces of $F$ containing the subset~$S$. An
unspecified open subset of $\Dou(\Ii_S)$ containing the point
corresponding to~$C$ (strictly speaking, to the ideal sheaf of~$C$,
which is a subsheaf of $\Ii_S$) and isomorphic to an analytic space,
will be denoted by~$\Do(F,C,S)$. In view of the natural embedding of
$\Dou(\Ii_S)$ into $\Dou(\Oo_F)=\Do(F)$ we will always assume that
$\Do(F,C,S)\subset\Do(F,C)$.

Extending Notation~\ref{not:C_a},
we will denote by $C_a\subset F$ the analytic subspace of~$F$
corresponding to the point~$a\in\Do(F,C,S)$. 

Let $a\in \Do(F,C)$ be the point corresponding to $C\subset F$, and let
$\Ii_C\cong \Oo_F(-C)$ be the ideal sheaf of~$C\subset F$.
According to the general theory, the Zariski tangent space to $\Do(F,C)$
at~$a$ is 
\begin{equation}\label{eq:h^0(N)}
T_a\Do(F,C)=\Hom_{\Oo_F}(\Ii_C,\Oo_F/\Ii_C)\cong
\Hom_{\Oo_C}(\Ii_C/\Ii_C^2,\Oo_C)=\Nc_{F|C}
\cong\Oo_C(d)  
\end{equation}
(here and below, $\Hom$ refers to the space of global homomorphisms, not
to the $\Hom$~sheaf).
Similarly, taking into account that $\Ii_C\subset\Ii_S$ and denoting
by $b\in\Do(F,C,S)$ the point corresponding to~$C$, one has
\begin{multline*}
T_b\Do(F,C,S)\cong \Hom_{\Oo_F}(\Ii_C,\Ii_S/\Ii_C)\\{}\cong
\Hom_{\Oo_C}(\Ii_C/\Ii_C^2,\Ii_S)\cong\Nc_{F|C}\otimes\Oo_C(-S)\cong\Oo_C(d-m).
\end{multline*}

The main results about deformations of $C\subset F$ that we need are as
follows.

\begin{proposition}\label{def1}
Suppose that $F$ is a smooth and connected complex surface, $C\subset
F$ is a complex submanifold isomorphic to $\PP^1$, and $d=(C\cdot
C)\ge 0$. Then there exists a neighborhood~$\Delta\ni a$ of the
point~$a\in\Do(F,C)$ 
corresponding to~$C$ such that the analytic space $\Delta$ is a
smooth complex manifold of dimension~$d+1$ and, for any $b\in\Delta$,
$C_b\cong\PP^1$.
\end{proposition}

A similar result, of which Proposition~\ref{def1} is a
particular case, holds for $\Do(F,C,S)$.

\begin{proposition}\label{def2}
In the above setting, suppose that $S=\{\lst pm\}\subset C$ is a
subset of cardinality $m\le d$. Then there exists a neighborhood
$\Delta\ni a$ of the point~$a\in\Do(F,C,S)$ corresponding to~$C$ such
that the analytic space $\Delta$ is a smooth complex manifold of
dimension~$d-m+1$ and, for any $b\in\Delta$, $C_b\cong\PP^1$.
\end{proposition}

We begin with a particular case, which is essentially contained
in~\cite{Savelyev} (and which follows from the main result of~\cite{Kodaira}). 

\begin{lemma}\label{Sav}
Proposition~\ref{def1} holds if $d=0$.  
\end{lemma}

\begin{proof}
Let $a\in\Do(F,C)$ be the point corresponding to $C\subset F$. Since
$d=0$, one has $\Nc_{F|C}\cong\Oo_C$, so it follows
from~\eqref{eq:h^0(N)} that $\dim T_a\Do(F,C)=1$. But, according to
the main result of~\cite{Savelyev}, there exist a neighborhood
$W\supset C$ and an isomorphism $\ph\colon W\to D\times C$, where $D$
is the unit disk in~\C, such that $\ph(p)=(0,p)$ for any $p\in C$. If
one puts
\begin{equation*}
\Gf=\{(z,x)\in D\times F\colon x\in W,\ \pr_1(\ph(x))=z\},  
\end{equation*}
then the family \Gf induces a morphism $\Phi\colon D\to\Do(F,C)$ such
that $\Phi(0)=a$ (the point corresponding to~$C$) and $\Phi$ is 1--1
onto its image. Hence, $\dim_a\Do(F,C)\ge 1$. Since $\dim
T_a\Do(F,C)=1$, one concludes that $\Do(F,C)$ is a smooth
$1$-dimensional complex manifold in a neighborhood of~$a$.
\end{proof}

To prove Proposition~\ref{def2} in full generality, we will need two
simple lemmas. 

\begin{lemma}\label{nearby_fibers}
If $p\colon \Ho\to D$, where $D$ is the unit disk in~\C,
is a proper and flat morphism of analytic spaces, and if the fiber
$p^{-1}(0)$ is reduced and isomorphic to~$\PP^1$, then there exists an
$\eps\in(0;1)$ such that the fiber $f^{-1}(a)$ is also reduced and
isomorphic to $\PP^1$ whenever~$|a|<\eps$.
\end{lemma}

\begin{proof}[Sketch of proof]
It is easy to see that there exists an $\eps>0$ such that
$p^{-1}(D_\eps)\to D_\eps$ is a proper submersion of complex
manifolds. Hence, topologically it is a locally trivial bundle, so all
the fibers are homeomorphic to $S^2$, whence the
result.  
\end{proof}

\begin{lemma}\label{lemma:criterion}
  Suppose that $X$ is an analytic space, $a\in X$, and $\dim
  T_aX=n$. Then the following two assertions are equivalent.

\textup{(1)} $X$ is a smooth $n$-dimensional complex manifold in a
neighborhood of the point~$a$.

\textup{(2)} There exists a non-empty Zariski open subset
$V\subset\PP(T_aX)$ such that for any $1$-dimensional linear subspace
$\ell\subset T_aX$ corresponding to a point of~$V$ there exists a
smooth $1$-dimensional locally closed complex submanifold $Y\subset X$
such that $Y\ni a$ and $T_aY=\ell\subset T_aX$.
\end{lemma}

\begin{proof}
Only the implication $(2)\Rightarrow (1)$ deserves a proof.

Observe that $X$ is a complex manifold near~$a$ if and only if
$\dim_aX=n=\dim T_aX$. Furthermore, 
the question being local, we may and will assume that $X$ is a closed
analytic subspace of a polydisc $D\subset \C^N$. Let $\bar D$ be the
blowup of $D$ at~$a$, and let $\bar X$ be the strict transform
of~$X_{\mathrm{red}}$. If $\sigma\colon \bar D\to D$ is the blowdown
morphism and $E=\sigma^{-1}(a)$ is the exceptional divisor, then
$\bar X\cap E$ is a projective submanifold of~$E\cong\PP^{N-1}$,
$\dim\bar X\cap E=\dim_aX-1$, and $\bar X\cap E\subset\PP(T_aX)\subset
E$.

Now if $Y\subset X$ is a locally closed $1$-dimensional complex
submanifold such that $Y\ni a$ and if $\ell=T_aY\subset T_aX$, then
the point of~$\PP(T_a(X))$ corresponding to $\ell$ belongs to $\bar
X\cap E$; thus, it follows from (2) that $\bar X\cap E$ contains a
non-empty Zariski open subset of~$\PP(T_aX)$, hence $\bar X\cap
E=\PP(T_aX)$, hence $\dim_aX=n$, and we are done.
\end{proof}

\begin{proof}[Proof of Proposition~\ref{def2}]
Choose $d-m$ distinct points $\lst q{d-m}\in C\setminus S$. Let $\bar
F$ be the blowup of $F$ at the points $\lst pm,\lst q{d-m}$, let
$\sigma\colon \bar F\to F$ be the corresponding blowdown morphism, and
let $\bar C\subset \bar F$ be the strict transform of~$C$. One has
$\bar C\cong\PP^1$ and $(\bar C,\bar C)=0$. Let $\bar a\in \Do(\bar
F,\bar C)$ be the point corresponding to~$\bar C$, and let
$a\in\Do(F,C,S)$ be the point corresponding to~$C$.

Applying
Lemma~\ref{Sav} to the pair $(\bar F,\bar C)$, one concludes that
there exists a family $\bar\Ho_0\subset D\times \bar F$, where $D$ is the
unit disk in the complex plane, such that its fiber over~$0$ is $\bar
C\subset\bar F$ and, for the induced mapping $\bar\ph\colon D\to \Do(\bar
F,\bar C)$, its derivative $D\bar\ph(0)\colon T_0D\to T_{\bar
  a}\Do(\bar F,\bar C)$ is non-degenerate.

If we put
$\Ho_0=(\mathrm{id}\times\sigma)(\bar \Ho_0)\subset D\times F$,  
then $\Ho_0$ is a family of analytic subspaces in~$F$ containing~$S$;
its fiber over~$0$ is~$C$. Let $\ph\colon D\to \Do(F,C,S)$ be the
mapping induced by this family.

It is clear that the diagram
\begin{equation*}
  \xymatrix{
    &{T_{\bar a}\Do(\bar F,\bar C)}   \ar@{=}[r]
  &{H^0(\Nc_{\bar F|\bar C})}   \ar[dd]^{D\sigma}
  \\
    {T_0D} \ar[ur]^{D\bar\ph(0)}\ar[dr]_{D\ph(0)}\ar[urr]_{\bar\alpha}
                                        \ar[drr]^{\alpha}
    \\
     & {T_{a}\Do(F,C,S)} \ar@{=}[r]
     &{H^0(\Nc_{F|C}(-S)),}
}    
\end{equation*}
where the vertical arrow is induced by the natural homomorphism
$\Nc_{\bar F|\bar C}\to \sigma^*\Nc_{F|C}$, is commutative. It follows
from (the proof of) Lemma~\ref{Sav} that
$\bar\alpha(\partial/\partial z)$, where $z$ is the coordinate on~$D$,
is a nowhere vanishing section of $\Nc_{\bar F|\bar C}\cong\Oo_{\bar
  C}$; since the derivative of the mapping~$\sigma$ is non-degenerate
outside $\sigma^{-1}\{\lst pm,\lst q{d-m}\}$, the section
$\alpha(\partial/\partial z)=D\sigma(\bar\alpha (\partial/\partial
z))$ is not identically zero. Hence, $\ph$ induces an embedding of a
possibly smaller disk $D_\eps\subset D$ in $\Do(F,C,S)$.

Moreover, since $\sigma$ maps each of the curves $\sigma^{-1}(p_i)$,
$\sigma^{-1}(q_j)$ to a point, and since each of these curves is
transverse to $\bar C$, the section $\alpha(\partial/\partial z)$
vanishes at $\lst pm,\lst q{d-m}$, so $\alpha(\partial/\partial z)$
spans the $1$-dimensional linear space
\begin{equation}\label{eq:1-dim}
  H^0(\Nc_{F|C}(-S)(-q_1-\dots-q_{d-m}))\subset H^0(\Nc_{F|C}(-S))
  =T_a\Do(F,C,S).
\end{equation}

In the argument that follows we will assume that $d-m\ge2$, so that the
words about Veronese curves in $\PP^d$ and $\PP^{d-m}$ make
sense; we leave it to the reader to modify the wording for the case
$d-m=1$. 

Keeping the above in mind, identify $C$ with $\PP^1$ and $\Nc_{F|C}$
with $\Oo_{\PP^1}(d)=\Oo_C(d)$, embed $C$ in $\PP^d$ with the complete
linear system $|\Oo_C(d)|$ to obtain a Veronese curve
$C_d\subset\PP^d$, and project $C_d$ from $\PP^d$ to $\PP^{d-m}$, the
center of projection being the linear span of the images of the points
\lst pm. The image of this projection will be a Veronese curve
$C_{d-m}\subset \PP^{d-m}$; denote the resulting isomorphism between
$C$ and $C_{d-m}$ by $\ph\colon C\to C_{d-m}$.

One has
\begin{equation*}
\PP^{d-m}=\PP((H^0(\Nc_{F|C}(-S))^*);
\end{equation*}
for any $d-m$ distinct points $\lst q{d-m}\in C\setminus S$ the
points $\ph(q_1),\dots,\ph(q_{d-m})$ span a hyperplane in $\PP^{d-m}$, and
the linear span $\langle \ph(q_1),\dots,\ph(q_{d-m})\rangle$ is the
projectivization of
\begin{equation*}
\Ann (H^0(\Nc_{F|C}(-S)(-q_1-\dots-q_{d-m}))) 
\subset H^0(\Nc_{F|C}(-S))^*.  
\end{equation*}

Now the hyperplanes in $\PP^{d-m}$ that are transverse to $C_{d-m}$
form a non-empty Zariski open subset in
$(\PP^{d-m})^*=\PP(H^0(\Nc_{F|C}(-S)))$ and any such hyperplane
intersects $C_{d-m}$ at $d-m$ distinct points that are of the form
$\ph(q_1),\dots, \ph(q_{d-m})$, where $\lst q{d-m}\in C\setminus S$. Thus, the
linear subspaces of the form~\eqref{eq:1-dim} fill a non-empty Zariski open
subset in $\PP(T_a\Do(F,C,S))$  as we vary $\lst q{d-m}\in C\setminus
S$, so the hypotheses of Lemma~\ref{lemma:criterion} are satisfied if one
puts $Y=\ph(D)$, hence
the smoothness is established. Now the assertion to the effect that
the fibers are isomorphic to $\PP^1$ follows from
Lemma~\ref{nearby_fibers}. 
\end{proof}

\section{Good neighborhoods}\label{sec:good}

Suppose that $F$ is a smooth and connected complex surface and
$C\subset F$ is a curve that is isomorphic to $\PP^1$ and~$(C\cdot
C)\ge 0$. In the definition below we use Notation~\ref{not:C_a}.

\begin{definition}\label{def:good}
 We will say that an open subset $W\subset F$, $W\supset C$ is a
 \emph{good neighborhood of~$C$} if there exists a connected open
 subset $\Delta\subset \Do(F,C)$ such that $W=\bigcup_{b\in\Delta}C_b$
 and each $C_b$ for $b\in\Delta$ is a smooth curve isomorphic
 to~$\PP^1$.
\end{definition}

\begin{proposition}\label{prop:good_neighb}
In the above setting, there exists a fundamental system of good
neighborhoods of~$C$.  
\end{proposition}

\begin{proof}
If $d=0$, it follows immediately from Lemma~\ref{Sav}. Suppose that
$d>0$. 

Let $a\in\Do(F)$ be the point corresponding to $C$, and let $\Delta\ni
a$, $\Delta\subset \Do(F,C)$ be the neighborhood whose existence is
asserted by Proposition~\ref{def1}. We denote by
$\Ho\subset\Delta\times F$ the family of analytic subspaces of~$F$
induced by the embedding $\Delta\hookrightarrow \Do(F,C)$ (informally
speaking, $\Ho=\{(b,x)\in\Delta\times F\colon x\in C_b\}$). Since all
the $C_b$'s are smooth $1$-dimensional complex submanifolds of~$F$ and
the base $\Delta$ is a smooth complex manifold as well, \Ho is a
smooth complex manifold.

  I claim that the projection $q\colon \Ho\to F$ is a submersion; once
  we have established this fact, it will follow that $q(\Ho)\subset F$
  is a good neighborhood of~$C$.

To check this submersivity, consider an arbitrary point $(b,x)\in\Ho$
(i.e., $x\in C_b)$; we are to show that the derivative $Dq(b,x)\colon
T_{(b,x)}\Ho \to T_xF$ is surjective. To that end, pick $d$ distinct
point $\lst pd\in C_b\setminus\{x\}$, where $d=(C\cdot C)=(C_b\cdot
C_b)$, and put $\{\lst pd\}=S$.

Let $i\colon \Do(F,C,S)\hookrightarrow \Do(F,C)$ be
the natural embedding, and let $\beta\in\Do(F,C,S)$ be the point
corresponding to the curve $C_b$  (so $C_\beta$ and $C_b$ are the same
curve in~$F$, and $i(\beta)=b$).

Let $\Delta_0\subset \Do(F,C_b,S)$,
$\Delta_0\ni\beta$ be a neighborhood whose existence is asserted by
Proposition~\ref{def2};
we may and will assume that $\Delta_0\subset
i^{-1}(\Delta)$. Finally, let $\Ho_0\subset \Delta_0\times F$ be the
family of analytic subspaces of $F$ (containing~$S$) induced by the
inclusion $\Delta_0\hookrightarrow \Do(F,C,S)$, and let $q_0\colon
\Ho_0\to F$ be the projection. The inclusion $\Delta_0\subset
\Do(F,C)$ induces an inclusion $\Ho_0\to\Ho$, and one has the
following obvious commutative diagram:
\begin{equation}\label{diag}
  \xymatrix{
    &&F\\
    {\Ho_0}\ar@{^{(}->}[r]\ar[rru]^{q_0}\ar[d]&{\Ho}\ar[ru]_q\ar[d]\\
    {\Delta_0}\ar@{^{(}->}[r]^i&\Delta
 }   
\end{equation}

Observe that if $c\in i(\Delta_0)\subset \Delta$ and $y\in
C_c\setminus S$ then there exists a unique $\gamma\in \Delta_0$ such
that $y\in C_\gamma$. Indeed, if $y\in C_\gamma\cap C_{\gamma_1}$,
$\gamma,\gamma_1\in\Delta_0$, then $C_\gamma\cap C_{\gamma_1}\supset
\{y\}\cup S$, whence $(C_\gamma\cdot C_{\gamma_1})\geqslant d+1$, which
contradicts the fact that
\begin{equation*}
(C_\gamma\cdot C_{\gamma_1})=(C\cdot C)=d.  
\end{equation*}
Hence, there exists a neighborhood $\Vv\ni(\beta,x)$ in $\Ho_0$ such
that the restriction of $q_0$ to \Vv is 1--1 onto its image.

Since, according to Proposition~\ref{def2}, the Douady space $\Do(F,C_b,S)$ is
1-dimensional and smooth in a neighborhood of the point $\beta$ corresponding
to~$C_b\subset F$, and since a holomorphic mapping of complex
manifolds of the same dimension that is 1--1 onto its image is an open
embedding, it follows now that $q_0(\Vv)$ is open in~$F$ and
$q_0|_\Vv\colon \Vv\to q_0(\Vv)$ is a biholomorphism. In particular,
$Dq_0(\beta,x)\colon T_{(\beta,x)}\Ho_0\to T_xF$ is an isomorphism. Now
it follows from the diagram~\eqref{diag}
that $Dq(b,x)\colon T_{(b,x)}\Ho \to T_xF$ is a surjection.

This proves the submersivity of the projection $q\colon\Ho\to F$, so
$q(\Ho)\subset F$ is a good neighborhood of~$C$, and, for any open and
connected $\Delta'\subset\Delta$, $\Delta'\ni a$, the set
$q(p^{-1}(\Delta')) \supset C$, where $p\colon \Ho\to\Delta$ is the
projection, is a good neighborhood of~$C$ as well. The neighborhoods
$q(p^{-1}(\Delta'))$, for various such~$\Delta'$, form a fundamental
system of good neighborhoods of~$C$.
\end{proof}

\begin{proposition}\label{transverse}
Suppose that $C\subset F$, where $F$ is a smooth complex surface, is a
curve that is isomorphic to $\PP^1$, and that $(C\cdot C)=d>0$. If
$W\subset F$  is a good neighborhood of~$C$ in the sense of
Definition~\ref{def:good}, then for any $x\in W$ there exist two
curves $C_1,C_2\ni x$, $C_1,C_2\subset W$ such that
$C_1\cong C_2\cong\PP^1$ and the curves $C_1$ and $C_2$ are
transverse at~$x$.
\end{proposition}

\begin{proof}
Since $W$ is a good neighborhood, there exists a curve $C_b\ni x$,
where $b\in\Delta$. Pick $d-1$ distinct points $\lst p{d-1}\in
C_b\setminus\{x\}$ and put $S=\{x,\lst p{d-1}\}$. According to
Proposition~\ref{def2}, one has $\dim_\beta\Do(F,C_b,S)=1$, where $\beta$
is the point of $\Do(F,C_b,S)$ corresponding to~$C_b$. Let
$\Delta_0\ni\beta$, $\Delta_0\subset \Do(F,C_b,S)$ be the neighborhood
whose existence is asserted by Proposition~\ref{def2}, and let
$\gamma_1,\gamma_2\in\Delta_0$ be two distinct points. I claim that
the curves $C_1:=C_{\gamma_1}$ and
$C_2:=C_{\gamma_2}$ are transverse at~$x$. Indeed, if this not the case,
then the local intersection index of $C_{1}$ and $C_{2}$
at~$x$ is at least~$2$, whence
\begin{equation*}
(C\cdot C)=(C_{1}\cdot C_{2})  \ge d-1+2\ge d+1,
\end{equation*}
contrary to the fact that $(C\cdot C)=d$. This contradiction proves
the required transversality.
\end{proof}

\section{Transcendence degree 2}\label{sec:Castelnuovo}

In this section we prove 
Proposition~\ref{main}.
We begin with two simple observations.

\begin{proposition}\label{fin.gen}
If $F$ is a smooth connected complex surface that contains a
curve $C\subset\PP^1$ such that $(C\cdot C)>0$, then the field of
meromorphic functions $\M(F)$ is finitely generated over~\C and
$\trdeg_\C\M(F)\le 2$.
\end{proposition}

\begin{proof}
Theorem~2.1 from~\cite{Orevkov} asserts that there exists a
pseudoconcave neighborhood $U\supset C$. According to
\cite[Th\'eor\`eme 5]{Andreotti}, the field $\M(U)$ is finitely
generated over~\C and $\trdeg_\C\M(U)\le 2$. Observe that $\M(F)$ embeds in
$\M(U)$ as an extension of~\C; since any sub-extension of a finitely
generated extension of fields is also finitely generated and the transcendence
degree is additive in towers, we are done.
\end{proof}

\begin{lemma}\label{lemma:Lueroth}
  Suppose that $F$ is a connected complex surface such that
  $\trdeg_\C\M(F)\ge 2$. If there exists a connected open set $U\subset F$
  such that $\M(U)\cong \C(T_1,T_2)$, then $\M(F)\cong \C(T_1,T_2)$.
\end{lemma}

\begin{proof}
Since $\M(F)$ embeds into $\M(U)$, it follows immediately from
the two-dimensional L\"uroth theorem.   
\end{proof}
(Recall that the two-dimensional L\"uroth theorem asserts that if
$K\subset \C(T_1,T_2)$ is a subfield containing~\C and $\trdeg_\C K=2$,
then $K\cong \C(T_1,T_2)$; this fact follows immediately from
the existence of a smooth projective model for any finitely generated
extension of~\C and from
Theorem  3.5 in~\cite[Chapter VI]{CCS}.)

Now we may begin the proof of Proposition~\ref{main}. Thus, let $F$ be
a smooth connected complex surface such that $\trdeg_\C\M(F)\ge 2$ and
let $C\cong\PP^1$ be a curve (one-dimensional complex submanifold)
such that $(C\cdot C)=d>0$. Let $W\subset F$ be a good neighborhood
of~$C$ in the sense of Definition~\ref{def:good}. We are to prove that
$\M(F)\cong\C(T_1,T_2)$; by virtue of Lemma~\ref{lemma:Lueroth} it
suffices to prove that $\M(W)\cong\C(T_1,T_2)$.

Since $\trdeg_\C\M(F)\le 2$
by virtue of
Proposition~\ref{fin.gen} and $\trdeg_\C\M(F)\ge 2$ by hypothesis, one has
$\trdeg_\C\M(F)=2$. Now $\M(F)$ is isomorphic to a subfield of
$\M(W)$, so $\trdeg_\C\M(W)\ge 2$. Since Proposition~\ref{fin.gen} implies that
$\trdeg_\C\M(W)\le 2$ and $\M(F)$ is finitely generated over~\C, one
concludes that 
$\M(W)$ is a finitely generated extension of~\C, of transcendence
degree~$2$. Hence, $\M(W)=\C(f,g,h)$, where the meromorphic functions $f$
and~$g$ are algebraically independent over~\C and $h$ is algebraic
over~$\C(f,g)$ (of course, if one may set $h=0$, there is nothing to
prove). Denote by $P$ an irreducible polynomial in three independent
variables $F$, $G$, and~$H$ such that $P(f,g,h)$ is identically zero.

Now let $Y\subset\C^3$ be the affine algebraic surface that is the
zero locus of~$P$, and let $X\subset\PP^N$ be a smooth projective
model of~$Y$. 

Denote by $V\subset W$ the open subset on which each of
the meromorphic functions $f$, $g$, and $h$ is well defined and
consider the holomorphic mapping $\Phi\colon V\to Y$ defined by the
formula $x\mapsto (f(x),g(x),h(x))$. The mapping~$\Phi$ extends to a
meromorphic mapping from $W$ to $\bar Y\subset\PP^3$, where $\bar Y$
is the closure of $Y$; composing this meromorphic mapping with a
birational mapping $\bar Y\dasharrow X$, one obtains a meromorphic
mapping $\Phi_1\colon W\dasharrow X$.

\begin{lemma}\label{lemma:O}
There exists a non-empty open subset $O\subset W$ such that $\Phi_1$
is defined on~$O$ and the derivative $D\Phi_1(x)$ is non-degenerate
for any~$x\in O$.  
\end{lemma}

Assuming this lemma for a while, let us finish the proof of
Proposition~\ref{main}.

Our construction of the surfaces $X$ and~$Y$ implies that
$\M(X)\cong\M(W)$; hence, to prove
Proposition~\ref{main} it suffices to show that
$\M(X)\cong\C(T_1,T_2)$. We will derive this fact from the Castelnuovo
rationality criterion (see for example \cite[Chapter VI,
  3.4]{CCS}), which may be stated as follows.

\begin{theorem}[Castelnuovo]\label{Castelnuovo}
Suppose that $X$ is a smooth projective surface over~\C. Then
$\M(X)\cong \C(T_1,T_2)$ if and only if $H^0(X,\Omega^1_X)=0$ and
$H^0(X,\omega_X^{\otimes2})=0$.   
\end{theorem}

Here, $H^0(X,\Omega^1_X)$ is the space of holomorphic 1-forms on~$X$
and $H^0(X,\omega_X^{\otimes2})$ is the space of holomorphic 2-forms
of weight~$2$ on~$X$; we are to check that, for our surface $X$, both
these linear spaces do not contain non-zero elements.

To begin with, observe that if $\eta$ is a holomorphic covariant tensor field
on~$X$ that is not identically zero, then $\Phi_1^*\eta$ is a
holomorphic tensor field on $W\setminus I$, where $I$ is the
indeterminacy locus of~$\Phi_1$, and, in view of Lemma~\ref{lemma:O},
$\Phi_1^*\eta$ is not identically zero. Since $I$ is a discrete subset
of the complex surface~$W$, $\Phi_1^*\eta$ extends to a tensor field on
the entire~$W$. Thus, to show that $H^0(X,\Omega^1_X)=0$ and
$H^0(X,\omega_X^{\otimes2})=0$ it suffices to show that
$H^0(W,\Omega^1_W)=0$ and $H^0(W,\omega_W^{\otimes2})=0$, that is,
that there are no non-trivial holomorphic $1$-forms or $2$-forms of
weight~$2$ on~$W$. We deal with these two types of tensor fields
separately.

\textbf{The absence of holomorphic $1$-forms.} This is just the
following lemma, which will be used in Section~\ref{sec:trdeg=1} as
well.
\begin{lemma}\label{/omega^1=0}
Suppose that $F$ is a non-singular complex surface, $C\subset F$ is a
curve such that $C\cong\PP^1$ and $(C\cdot C)>0$, and $W$ is a good
neighborhood of $C$. Then any holomorphic $1$-form on~$W$ is
identically zero.
\end{lemma}

\begin{proof}
  Suppose that $\omega$ is such a form.  Proposition~\ref{transverse}
  implies that, for any $x\in W$ there exist curves $C_1,C_2\subset
  W$, $C_1\cong C_2\cong\PP^1$ such that $C_1\cap C_2\ni x$ and $C_1$
  and $C_2$ are transverse at~$X$. Since there are no non-zero
  holomorphic $1$-forms on the Riemann sphere, the restriction of
  $\omega$ to both $C_1$ and~$C_2$ is identically zero. Hence, the
  linear functional $\omega_x$ that $\omega$ induces on $T_xW$ is zero
  on $T_xC_1,T_xC_2\subset T_xW$. Since $C_1$ and $C_2$ are transverse
  at~$x$, these two linear spaces span the entire $T_xW$, so
  $\omega_x=0$. Since $x\in W$ was arbitrary, $\omega=0$ and we are
  done.
\end{proof}

\textbf{The absence of holomorphic $2$-forms of weight~$2$.}  Recall
that differential $2$-forms of weight~$2$ on a surface~$G$ have, in
local coordinates~$(z,w)$, the form $f(z,w)(dz\wedge dw)^2$.  If
$\omega$ is such a form, then, for any point~$x\in G$, $\omega_x$ is a
mapping from $T_xG\times T_xG$ to \C; this mapping is uniquely
determined by its value at a given pair of linearly independent
tangent vectors.

\begin{lemma}\label{U/times/P^1}
If $G=U\times\PP^1$, where $U$ is an open subset of~\C, then there is
no non-trivial holomorphic $2$-form of weight~$2$ on~$G$. 
\end{lemma}

\begin{proof}
Suppose that $\omega$ is such a form. If $z$ is the coordinate on
$U\subset\C$, then there exists a nowhere vanishing holomorphic vector
field~$\partial/\partial z$ on~$G$. For any $b\in U$, put
$C_b=\{b\}\times\PP^1$. To show that $\omega=0$ it suffices to show
that for any $b\in U$ and $x\in C_b$ one has
$\omega_x(\partial/\partial z,v)=0$, where $v\in T_xC_b$ is a nonzero
tangent vector.

Consider the tensor field $\eta=i_{\partial/\partial z}\omega$ (the
contraction of $\omega$ with $\partial/\partial z$), which is a family
of functions $\eta_x\colon T_xG\to\C$ for all $x\in G$,
$\eta_x(w)=\omega_x(\partial/\partial z,w)$ for $w\in T_xG$. The field
$\omega$ is a holomorphic section of $\Sym^2\Omega^1_G$, and its
restriction to each $C_b$ is a section of $\omega_{C_b}^{\otimes2}$,
i.e., a quadratic differential on~$C_b$, i.e., a section of
$\Oo_{C_b}(-4)$; such a holomorphic section must be identically zero,
so, for any $x\in C_b$ and any $v\in T_xC_b$,
$\omega_x(\partial/\partial z,v)=\eta(v)=0$, and we are done.
\end{proof}

Now suppose that $\omega$ is a differential $2$-form of
weight~$2$ on~$W$. To show that $\omega=0$, pick $d$ distinct points
$\lst pd\in C$, where $d=(C\cdot C)$, and let $\bar W$ be the blowup
of $W$ at \lst pd and $\bar C\subset\bar W$ be the strict transform
of~$C$.  It suffices to show that $\sigma^*\omega=0$, where
$\sigma\colon \bar W\to W$ is the blowdown morphism, and it will
suffice to show that
$\sigma^*\omega=0$ on a non-empty open subset of~$\bar W$. Since $\bar
C\cong\PP^1$ and $(\bar C\cdot\bar C)=0$, it follows from the main
result of Savelyev's paper~\cite{Savelyev} that a neighborhood of $\bar
C$ in $\bar W$ is isomorphic to $U\times\PP^1$, where $U$ is an open
subset of~$\C$. Now Lemma~\ref{U/times/P^1} applies.

This completes the proof of Proposition~\ref{main} modulo
Lemma~\ref{lemma:O}. 
\begin{proof}[Proof of Lemma~\ref{lemma:O}]
It suffices to prove this assertion for~$\Phi\colon V\to Y\subset
\C^3$ instead of~$\Phi_1$.  Moreover, if $\pi\colon Y\to\C^2$ is the
projection defined by forgetting the third coordinate, then the
derivative of $\pi$ is non-degenerate on a non-empty Zariski open
subset of the smooth locus of~$Y$; hence, it suffices to establish the
existence of such a set~$O$ for the mapping $\Psi=\pi\circ\Phi\colon
V\to\C^2$, $\Psi\colon x\mapsto (f(x),g(x))$.

The mapping $\Psi$ extends to a meromorphic mapping $W\dasharrow
\PP^2$ defined, in the homogeneous coordinates, by $x\mapsto
(1:f(x):g(x))$; abusing the notation, we will denote this meromorphic
mapping by~$\Psi$ as well. The indeterminacy locus of the meromorphic
mapping $\Psi$ is a discrete subset of~$W$.

If there exists at least one point $x\in W$ where $\Psi$ is defined
and $D\Psi(x)$ is non-degenerate, we are done. Assume now that the
derivative of $\Psi$ is degenerate at any point where $\Psi$
is determined; we will show that this assumption leads to a
contradiction.

Let $\Delta\subset \Do(F,C)$ be the open subset such that
$W=q(p^{-1}(\Delta))$, where $p\colon \Ho\to \Do(F)$ and $q\colon
\Ho\to F$ are the canonical projections of the restriction of the
universal family~\Ho; recall that the curve $q(p^{-1}(b))\subset F$,
where $b\in\Delta$, is denoted by~$C_b$.

Observe that the restriction of $\Psi$ to any $C_b$ is a meromorphic,
hence holomorphic, mapping from~$C_b$ to $\PP^2$. For any $b_1,b_2\in
\Delta$, $(C_{b_1}\cdot C_{b_2})=(C\cdot C)>0$, hence $C_{b_1}\cap
C_{b_2}\ne \varnothing$. Thus, if the restriction of $\Psi$ to each
$C_b$ is constant, then $\Psi$ is constant, which is nonsense. Hence,
we may and will
pick a $b\in\Delta$ such that the restriction of $\Psi$
to $C_b$ is not constant. Put $\Psi(C_b)=Z\subset\PP^2$; it follows
from the Chow theorem that $Z$ is a projective algebraic curve.

Observe as well that the set of points $x\in W$ where $\Psi$ is
defined and $D\Psi(x)=0$ must have empty interior (otherwise $\Psi$
would be constant). Hence, there exists a closed analytic subset $D$
with empty interior such that, for any $x\in W\setminus D$, $\Psi$ is
defined at~$x$ and $\rank D\Psi(x)=1$. Hence, all the fibers of the
restriction $\Psi|_{W\setminus D}$ are smooth analytic curves in
$W\setminus D$.

\begin{figure}
  \centering
  \includegraphics{fol.mps}
  \caption{}\label{figure}
  
\end{figure}

Pick a point $x\in C_b\setminus D$ such that $T_xC_b$ is not contained
in $\Ker D\Psi(x)$. There exists an open set $U\ni x$, $U\subset
W\setminus D$ such that for any $y\in C_b\cap U$ the set
$\Psi^{-1}(\Psi(y))$ is a smooth analytic curve transverse to~$C$
at~$y$.  Now for any $b'\in\Delta$ that is close enough to~$b$ there
exists a non-empty open set $V\subset C_{b'}\cap U$ such that for any
$x'\in V$ there exists a point $y'\in C_b\cap U$ such that
$\Psi^{-1}(\Psi(y'))\cap C_{b'}$ contains~$x'$ (see Fig.~\ref{figure}).

Hence, $\Psi(V)\subset \Psi(C_b)$; since $V$ is a non-empty open
subset of~$C_{b'}$, one concludes that $\Psi(C_{b'})=\Psi(C_b)=Z\subset
\PP^2$. Since the curves $C_{b'}$, for all $b'$ close enough to~$b$,
sweep, by virtue of Proposition~\ref{prop:good_neighb}, an open subset
of~$W$, one concludes that $\Psi(W)\subset Z$. Since $Z$ is an
algebraic curve in $\PP^2$ and $\Psi$ is defined by the formula
$x\mapsto(1:f(x):g(x))$, it follows that the meromorphic functions $f$
and~$g$ are algebraically dependent, which yields the desired
contradiction.
\end{proof}

\section{Transcendence degree 1}\label{sec:trdeg=1}

In this section we prove Proposition~\ref{main1}. Its proof is
similar to that of Proposition~\ref{main}, but simpler.

To wit, by virtue of Proposition~\ref{fin.gen} the field $\M(F)$ is
finitely generated over~\C. Since $\trdeg_\C\M(F)=1$, one has
$\M(F)=\C(f,g)$, where the meromorphic functions~$f$ and~$g$ are
algebraically dependent over~\C (if $\M(F)$ is generated by one
function, there is nothing to prove). Denote by $P$ an irreducible
polynomial in two independent variables $F$ and~$G$ such that
$P(f,g)=0$; let $Y\subset\C^2$ be the affine curve that is the
zero locus of~$P$, and let $X$ be the smooth projective curve (aka
compact Riemann surface) for which $\M(X)\cong \M(Y)$.

Denote by $V\subset F$ the open subset on which both
$f$ and $g$ are well defined and
consider the holomorphic mapping $\Phi\colon V\to Y$ defined by the
formula $x\mapsto (f(x),g(x))$. The mapping~$\Phi$ extends to a
meromorphic mapping from $F$ to $\bar Y\subset\PP^2$, where $\bar Y$
is the closure of $Y$; composing this meromorphic mapping with a
birational mapping $\bar Y\dasharrow X$, one obtains a meromorphic
mapping $\Phi_1\colon F\dasharrow X$. Since, by our construction,
$\M(F)\cong \M(X)$, it suffices to show that $X\cong\PP^1$, or,
equivalently, that there are no non-trivial holomorphic $1$-forms
on~$X$. 

To that end, let $I\subset F$ be the
indeterminacy locus of~$\Phi_1$; it is a discrete subset
of~$F$. Choose a good neighborhood $W\supset C$; since $\Phi_1$ is not
constant, there exists a non-empty open
subset $O\subset W\setminus I$ such that $\rank D\Phi_1(x)=1$ for any
$x\in O$. Now if $\omega\ne0$ is a holomorphic form on $X$, then
$(\Phi_1|_{W\setminus I})^*\omega$ is a holomorphic form such that its
restriction to~$O$ is not identically zero. Extending it to~$W$, one
obtains a holomorphic $1$-form on $W$ which is not identically
zero. This contradicts Lemma~\ref{/omega^1=0}, and this contradiction
completes the proof of Proposition~\ref{main1}.

\bibliographystyle{amsplain}
\bibliography{field}

\end{document}